\documentclass[a4paper, 12pt]{amsart}
\usepackage{amssymb}
\usepackage{amsthm}
\usepackage{amsmath}
\usepackage{tikz-cd}
\usepackage{enumitem}
\usepackage{graphicx}
\usepackage[utf8]{inputenc}
\usepackage[english]{babel}

\newcommand*\circled[1]{\tikz[baseline=(char.base)]{
            \node[shape=circle,draw,inner sep=2pt] (char) {#1};}}

\DeclareMathOperator{\HS}{\mathrm{HS}}
\DeclareMathOperator{\HF}{\mathrm{HF}}

\newtheorem{theorem}{Theorem}[section]
\newtheorem{corollary}[theorem]{Corollary}
\newtheorem{lemma}[theorem]{Lemma}
\newtheorem{conjecture}[theorem]{Conjecture}
\newtheorem{proposition}[theorem]{Proposition}

\theoremstyle{definition}
\newtheorem{definition}[theorem]{Definition}
\newtheorem{example}[theorem]{Example}

\author{Filip Jonsson Kling}
\title[The SLP for Quadratic Reverse Lexicographic Ideals]{The Strong Lefschetz Property for Quadratic Reverse Lexicographic Ideals}

\address{Department of Mathematics\\ Stockholm University\\ SE-106 91 Stockholm, Sweden}
\email{filip.jonsson.kling@math.su.se}

\begin{document}

\maketitle

\begin{abstract}
Consider ideals $I$ of the form
\[
I=(x_1^2,\dots, x_n^2)+\mathrm{RLex}(x_ix_j)
\]
where $\mathrm{RLex}(x_ix_j)$ is the ideal generated by all the square-free monomials which are greater than or equal to $x_ix_j$ in the reverse lexicographic order. We will determine some interesting properties regarding the shape of the Hilbert series of $I$. Using a theorem of Lindsey \cite{Lindsey}, this allows for a short proof that any algebra defined by $I$ has the strong Lefschetz property when the underlying field is of characteristic zero. Building on recent work by Phuong and Tran \cite{New_Stanley}, this result is then extended to fields of sufficiently high positive characteristic. As a consequence, this shows that for any possible number of minimal generators for an artinian quadratic ideal there exists such an ideal minimally generated by that many monomials and defining an algebra with the strong Lefschetz property.
\end{abstract}

\section{Introduction}
Let $R=k[x_1,\dots, x_n]$ for some field $k$ and $A=R/I$ be a standard graded artinian algebra. A common question regarding such algebras has been in which cases they enjoy the weak or strong Lefschetz property (WLP or SLP, respectively). See for example the survey article \cite{Survey}.
Recall that $A$ satisfies the SLP if there exists a linear form $\ell\in A_1$ such that the multiplication by powers of that linear form
\[
\cdot \ell^i:A_j \to A_{j+i}
\]
has full rank for all $i,j\geq 0$. If we only require it to hold for $i=1$ and all $j\geq 0$, we say it satisfies the WLP. A classical result of this type is the case when $k$ is a field of characteristic zero and $A=R/I$ is defined by a monomial complete intersection, $I=(x_1^{d_1},\dots, x_n^{d_n})$. This was first shown independently by Stanley \cite{Stanley} and  Watanabe \cite{Watanabe} to have the SLP, with later proofs given by Reid, Roberts and Roitman \cite{RRR}, Hara and Watanabe \cite{Watanabe_Boolean}, Lindsey \cite{Lindsey}, and Phuong and Tran \cite{New_Stanley}. We will here use techniques from the two most recent proofs in \cite{Lindsey} and \cite{New_Stanley} to establish the SLP for a new family of algebras defined by quadratic ideals. In particular, this will show that for any possible number of minimal generators of an artinian quadratic ideal, there is a monomial ideal $I$ with that many generators defining an algebra with the SLP. This can be seen as a partial converse to a classification result of Altafi and Lundqvist \cite{Forcing_WLP}: instead of asking what number of generators of an artinian algebra forces it to have the WLP, we ask what number of generators forces it to \emph{fail} the WLP, and show that no such number exists in the quadratic case. This complements other recent work on the Lefschetz properties for ideals defined by quadratic monomials such as \cite{Dao_Nair}, \cite{WLP_for_quadratic} and \cite{Tran_WLP_paths+cycles}. 

The algebras in question are defined by ideals of the form 
\[
I=(x_1^2,\dots, x_n^2)+\mathrm{RLex}(x_ix_j)
\]
for $1\leq i<j\leq n$ where $\mathrm{RLex}(x_ix_j)$ is the ideal generated by all the square-free monomials which are greater than or equal to $x_ix_j$ in the reverse lexicographic order with $x_1>\cdots >x_n$. Recall that this order means that $x_1^{a_1}\cdots x_n^{a_n}>x_1^{b_1}\cdots x_n^{b_n}$ for two monomials of the same degree if the last non-zero entry of the vector $(a_1,\dots, a_n)-(b_1,\dots, b_n)$ is negative. Our main result is that $A=R/I$ always satisfies the SLP under some restrictions on the characteristic of $k$. We will also determine some properties of the Hilbert series of $A$ that are of independent interest.

\begin{example}
\label{ex:x_2x_4}
In $R=k[x_1,x_2,x_3,x_4]$, we have that 
\[
\mathrm{RLex}(x_2x_4)=(x_1x_2, x_1x_3, x_2x_3, x_1x_4, x_2x_4)
\] and we claim that for the ideal $I=(x_1^2, x_2^2,x_3^2,x_4^2)+ \mathrm{RLex}(x_2x_4)$, $R/I$ does have the SLP.    
\end{example}

One way to think about these ideals is to consider the graph associated to it, namely the graph $G$ such that $\mathrm{RLex}(x_ix_j)$ is the edge ideal of $G$. There we see that each $\mathrm{RLex}(x_{j-1}x_j)$ corresponds to the complete graph on $j$ vertices, and in general $\mathrm{RLex}(x_ix_j)$ is the complete graph on $j-1$ vertices together with $i$ edges connecting an additional vertex to the first $i$ vertices of that complete graph.  

\begin{figure}[h]
\begin{center}
\begin{tikzcd}
\circled{3} \arrow[d, no head] \arrow[r, no head] & \circled{1} \arrow[ld, no head] \arrow[d, no head] \\
\circled{2}                                       & \circled{4} \arrow[l, no head]                    
\end{tikzcd}
\caption{A graph with edge ideal given by $\mathrm{RLex}(x_2x_4)$.}
\end{center}
\label{fig:24}
\end{figure}
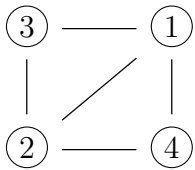

The structure for the rest of paper is that we begin by exploring properties of the Hilbert series of our algebras $A$, allowing for a short proof that they do have the SLP in characteristic zero. Next, the matrices representing multiplication by a Lefschetz element on $A$ are studied and the SLP is established when the underlying field is of sufficiently high characteristic. This will use our earlier obtained properties of the Hilbert series in a crucial way. Finally, we make some comments and a conjecture on what happens in higher degrees.

\section{The shape of its Hilbert series}

Let us begin by recalling the definition of a Hilbert series.

\begin{definition}
Let $A=\oplus_{i=0}^{\infty}A_i$ be a standard graded $k$-algebra. The \emph{Hilbert function} of $A$ is given by $\mathrm{HF}(A,i)=\dim_k(A_i)$. The \emph{Hilbert series} of $A$ is then
\[
\mathrm{HS}(A,t)=\sum_{i=0}^{\infty}\mathrm{HF}(A,i)t^i.
\]
We will also use the convention that $\mathrm{HF}(A,i)=\dim_k(A_i)=0$ for $i<0$.
\end{definition}

With this, we can already establish everything for our algebras in the simplest case, when $j=n$. 

\begin{lemma}
\label{lem:j=n case}
Let $I=(x_1^2,\dots, x_n^2) + \mathrm{RLex}(x_ix_n)$. Then $A=R/I$ has the Hilbert series $1+nt + (n-i-1)t^2$ and the SLP.
\end{lemma}

\begin{proof}
The Hilbert series for $A$ follows as the algebra is spanned by $x_1,\dots, x_n$ in degree $1$ and by $x_{i+1}x_n,\dots, x_{n-1}x_n$ in degree $2$ while it has nothing in higher degrees. It can then be checked that this has the SLP as the maps $\cdot\ell^i:A_0\to A_i$ always have full rank for $\ell=x_1+\cdots + x_n$ and $\cdot \ell:A_1\to A_2$ is surjective since $x_j\ell$ for $j=i+1,\dots, n-1$ will hit all basis elements in degree two.   
\end{proof}

Note that unless otherwise stated, all results in this section are independent of the characteristic of the underlying field. In particular, the algebras considered in Lemma \ref{lem:j=n case} have the SLP over any field. The case when $1\leq i <j<n$ is a bit trickier when it comes to determining if our algebras have the SLP. So before we can get to that, let us determine the Hilbert series of $A$ and find some of its properties.

\begin{theorem}
\label{the:Hilbert series}
The Hilbert series of $A=R/I$ where $I=(x_1^2,\dots, x_n^2)+\mathrm{RLex}(x_ix_j)$ is given by
\[
\HS(R/I;t) = (1+t)^{n-j}(1 + jt + (j-i-1)t^2).
\]
\end{theorem}

\begin{proof}
When $j=n$, we know that it holds from Lemma \ref{lem:j=n case}. So assume that $j<n$. Consider the following short exact sequence 
\[
0 \to R/(I:x_n) \xrightarrow{\cdot x_n} R/I \to R/(I+(x_n)) \to 0.
\]
As no generator of $\mathrm{RLex}(x_ix_j)$ is divisible by $x_n$, we find that $(I:x_n)=I+(x_n)=(x_1^2,\dots, x_{n-1}^2, x_n) + \mathrm{RLex}(x_ix_j)$. Therefore we get that 
\[
\HS(R/I;t) = (1+t)\HS(\bar{R}/\bar{I};t)
\]
where $\bar{R}=k[x_1,\dots, x_{n-1}]$ and $\bar{I}=(x_1^2,\dots, x_{n-1}^2) + \mathrm{RLex}(x_ix_j)$. If $j<n-1$, we apply the same argument again and after a total of $n-j$ steps we reach a polynomial ring with $j$ variables. But in that case we know that the Hilbert series equals $1 + jt + (j-i-1)t^2$ from Lemma \ref{lem:j=n case}, hence we find that
\[
\HS(R/I;t) = (1+t)^{n-j}(1 + jt + (j-i-1)t^2)
\]
as desired.
\end{proof}

\begin{corollary}
The Hilbert series $\HS(A;t)$ for $A$ as in Theorem \ref{the:Hilbert series} is a real-rooted polynomial. In particular, the coefficients form a unimodal and log-concave sequence \cite{Petter_Handbook}. Moreover, it is symmetric if and only if $j-i-1=1$.
\end{corollary}

\begin{proof}
We know that $(1+t)^{n-j}$ is real-rooted. To check that $1 + jt + (j-i-1)t^2$ is real-rooted, we look at its discriminant
\[
j^2 - 4(j-i-1) = (j-2)^2 + 4i > 0,
\]
which is always strictly positive, so $1 + jt + (j-i-1)t^2$ is real-rooted. Hence $\HS(A;t)$ is real-rooted as well. Next, if $j-i-1=1$, then $\HS(A;t)$ is a product of two symmetric polynomials and therefore also symmetric. If $j-i-1\neq1$, then we are in one of two cases. If $j-i-1>1$, then symmetry is broken already at the outermost coefficients. Similarly, if $j-i-1=0$, then again symmetry fails at the outermost coefficients as we assume that $j>i\geq 1$.
\end{proof}

We remark that the above results about Hilbert series for this kind of algebras also extend to quadratic monomial complete intersections if we interpret them as corresponding to the case when $j=2$ and $i=0$. The log-concavity result also adds to recent studies on the log-concavity of some families of Hilbert series by Iarrobino \cite{Iarrobino_log} and Zanello \cite{Zanello_log}. \\

The following is a key property that we later will see the coefficients of $\HS(R/I;t)$ satisfy.

\begin{definition}
Fix two indices $j<i$ with at least one of them in the interval $[0,n]$. A sequence of positive integers $(a_k)_{k=0}^n$ is \emph{mid-heavy} if for any such $i,j$, we have that $a_j\leq a_i$ implies $a_{j-1}\leq a_{i+1}$, and $a_j\geq a_i$ implies $a_{j-1}\geq a_{i+1}$. If $k<0$ or $k>n$, then the inequalities should be read as having $a_k=0$ for those $k$. 
\end{definition}

One may thus see the property of being mid-heavy not as a property of the finite sequence $(a_k)_{k=0}^n$, but rather as a property of a finitely supported sequence $(a_k)_{k\in \mathbb{Z}}$. This is related to the following class of Hilbert functions introduced by Lindsey \cite{Lindsey}.

\begin{definition}
Let $D$ be the socle degree of an artinian algebra $A$, that is, the highest degree $s$ for which $\dim_k(A_s)\neq 0$. Then the Hilbert series $\sum_ih_it^i$ of $A$ is in the class $\mathcal{H}$ if it satisfies 
\begin{equation}
\label{eq:First}
h_{i-1}\leq h_{D-i}\leq h_i \text{ for all } 1\leq i \leq \left\lfloor \frac{D}{2} \right\rfloor \text{ or} 
\end{equation}
\begin{equation}
\label{eq:second}
h_{D-i+1}\leq h_{i}\leq h_{D-i} \text{ for all } 1\leq i \leq \left\lfloor \frac{D}{2} \right\rfloor.
\end{equation}
\end{definition}

\begin{lemma}
\label{mid_to_H}
If the coefficients $h_i$ of a Hilbert series $\sum_ih_it^i$ form a mid-heavy sequence, then the same Hilbert series is in the class $\mathcal{H}$.    
\end{lemma}

\begin{proof}
Assume that 
\begin{equation}
\label{eq:start}
h_{D-\left\lfloor\frac{D}{2}\right\rfloor}\leq h_{\left\lfloor\frac{D}{2}\right\rfloor}.
\end{equation} 
We want to show that the sequence satisfies $\eqref{eq:First}$. If $D$ is odd, then the property of mid-heavy sequences gives that 
\begin{equation}
\label{eq:After_Start}
h_{D-\left(\left\lfloor\frac{D}{2}\right\rfloor-1\right)}\leq h_{\left\lfloor\frac{D}{2}\right\rfloor-1}.
\end{equation}
If $D$ is even, then \eqref{eq:start} is always true as $D-\left\lfloor\frac{D}{2}\right\rfloor = \left\lfloor\frac{D}{2}\right\rfloor$, and we take \eqref{eq:After_Start} as our starting point instead.
From here we can then independently of $D$ being even or odd use the defining property of mid-heavy sequences again and again to get that
\[
h_{D-i}\leq h_i \text{ for all } i=0,1,\dots, \left\lfloor\frac{D}{2}\right\rfloor.
\]
This gives one part of the inequality in \eqref{eq:First} that we need. To get the other part of the inequality, assume for a contradiction that $h_{i-1}\geq h_{D-i}$ for some $1\leq i \leq \left\lfloor\frac{D}{2}\right\rfloor$. Then the mid-heavy property gives that 
\[
h_{i-2} \geq h_{D-(i-1)},
\]
and a total of $i-1$ more applications of the property gives that 
\[
h_{-1}\geq h_D.
\]
But $h_{-1}=\dim_k(A_{-1})=0$ while $h_{D}>0$, giving the desired contradiction. The case for when $h_{D-\left\lfloor\frac{D}{2}\right\rfloor}\geq h_{\left\lfloor\frac{D}{2}\right\rfloor}$ is done similarly but establishing $\eqref{eq:second}$ instead. Hence the Hilbert series is in the class $\mathcal{H}$.
\end{proof}

Note that the reverse implication does not hold in general. For example $k[x]/(x^3)$ has the Hilbert series $1 + t + t^2$, so $h_0=h_1=h_2=1$, and it lies in the class $\mathcal{H}$, but $h_0,h_1,h_2$ is not a mid-heavy sequence as $h_0\geq h_1$ but $0=h_{-1}<h_2=1$.

The reason we introduced the class $\mathcal{H}$ is so we can use the following theorem of Lindsey \cite[Theorem 3.10]{Lindsey}.

\begin{theorem}
\label{thm:Lindsey}
Let $k$ be a field of characteristic zero and $\ell\in R_1$ a strong Lefschetz element for a graded $R$-module $M$. Then the Hilbert series of $M$ is in the class $\mathcal{H}$ if and only if $\ell + y$ is a strong Lefschetz element for $M\otimes_kk[y]/(y^m)$ for all $m\geq 0$.
\end{theorem}

When trying to determine if a sequence is mid-heavy or not, the following preservation result can be useful.

\begin{lemma}
\label{preservation}
Let $(a_k)_{k=0}^n$ be a mid-heavy sequence and let $p(t)$ be the corresponding generating function. Then $(1+t)p(t)$ is also a generating function for a mid-heavy sequence.
\end{lemma} 

\begin{proof}
The sequence associated to $(1+t)p(t)$ is $(a_{k-1} + a_k)_{k=0}^{n+1}$ where we recall that $a_k=0$ for $k<0$ and $k>n$. So given that $(a_{i-1} + a_i)\leq (a_{j-1} + a_j)$ for some $i<j$, we want to determine that $(a_{i-2} + a_{i-1})\leq (a_{j} + a_{j+1})$. 
Assume first that $j=i+1$. Then $(a_{i-1} + a_i)\leq (a_{j-1} + a_j)=(a_i + a_{i+1})$ is equivalent to $a_{i-1}\leq a_{i+1}$. Using that $(a_k)_{k=0}^n$ is mid-heavy, this gives that $a_{i-2}\leq a_{i+2}$, and thus 
\[
(a_{i-2} + a_{i-1})\leq (a_{i+1} + a_{i+2}) = (a_{j} + a_{j+1})
\]
holds. 
Next, assume that $j>i+1$. Then we claim that $(a_{i-1} + a_i)\leq (a_{j-1} + a_j)$ is equivalent to $a_i \leq a_{j-1}$. Indeed, if $a_i \leq a_{j-1}$, then also $a_{i-1}\leq a_j$ and the inequality is true. If on the other hand $a_i> a_{j-1}$, then $a_{i-1}\geq a_j$ and $(a_{i-1} + a_i)> (a_{j-1} + a_j)$, which is not the case and hence proves the claim. So given that $a_i \leq a_{j-1}$, we get $a_{i-1}\leq a_j$ from the mid-heavy property, which in turn gives that $a_{i-2}\leq a_{j+1}$. Hence
\[
(a_{i-2} + a_{i-1})\leq (a_{j} + a_{j+1}),
\]
which is our desired inequality. The case of the reverse inequalities is proved in exactly the same way, so we are done.
\end{proof}

Let us now use the machinery of mid-heavy sequences to gain information about our algebras $A=R/I$ for $I=(x_1^2,\dots, x_n^2)+\mathrm{RLex}(x_ix_j)$. 

\begin{corollary}
\label{cor:mid_heavy}
The coefficients of $\HS(A;t)$ form a mid-heavy sequence.
\end{corollary}

\begin{proof}
As $\HS(A;t) = (1+t)^{n-j}(1 + jt + (j-i-1)t^2)$ by Theorem \ref{the:Hilbert series}, the result follows from Lemma \ref{preservation} and the fact that $1, j, (j-i-1)$ is a mid-heavy sequence.
\end{proof}

\begin{theorem}
\label{thm:pf1}
Let $I=(x_1^2,\dots, x_n^2)+\mathrm{RLex}(x_ix_j)$ where $1\leq i < j \leq n$ and let $k$ be a field of characteristic zero. Then $A=R/I$ has the strong Lefschetz property.
\end{theorem}

\begin{proof}
We know that $A$ has the SLP when $n\leq 2$ by \cite[Proposition 4.4.]{Codim_2} as we work over a field of characteristic zero, as well as when $j=n$ by Lemma \ref{lem:j=n case}, so assume we are in the case $j<n$ with $n\geq 3$. We will prove the statement by induction on $n$. Recall that for $I=(x_1^2,\dots, x_n^2) + \mathrm{RLex}(x_ix_j)$, we have by definition that $\mathrm{RLex}(x_ix_j)$ is the ideal generated by all square-free monomials greater than or equal to $x_ix_j$ in the reverse lexicographic order. 
In particular, $I$ does not include any monomial of the form $x_kx_n$ for $k\neq n$. Now set $\bar{R}=k[x_1,\dots, x_{n-1}]$, $\bar{I}=(x_1^2,\dots, x_{n-1}^2) + \mathrm{RLex}(x_ix_j)$ and $\bar{A}=\bar{R}/\bar{I}$. By induction, we can assume that $\bar{A}$ has the SLP with strong Lefschetz element $\bar{\ell}=x_1+\cdots + x_{n-1}$. Now we can use that we know some of the structure of $\bar{A}$. By Corollary \ref{cor:mid_heavy}, the coefficients of its Hilbert function form a mid-heavy sequence. Lemma \ref{mid_to_H} then says that the Hilbert series is in the class $\mathcal{H}$. Finally, as we are assumed to be working over a field of characteristic zero, we can apply Theorem \ref{thm:Lindsey} to get that
\[
\bar{A}\otimes_kk[x_n]/(x_n^2) \cong A
\]
also has the SLP with strong Lefschetz element $\ell=x_1+\cdots + x_{n-1} + x_n$ as was to be shown.
\end{proof}

In \cite{Forcing_WLP}, Altafi and Lundqvist give sharp bounds on the number of minimal generators of an artinian ideal $I$ generated in a single degree $d$ which forces it to have the WLP. In the quadratic case $d=2$, we can now answer the reverse question, is there any bound on the number of minimal generators which forces $I$ to not have the WLP, in the negative.

\begin{corollary}
Let $R=k[x_1,\dots, x_n]$ where $k$ is a field of characteristic zero. Then for any $\mu \in \left[n,\binom{n+1}{2}\right]$, 
there is an artinian ideal $I$ minimally generated by $\mu$ quadratic monomials such that $R/I$ has the SLP.
\end{corollary}

\begin{proof}
If $\mu=n$ we can take $I$ to be a monomial complete intersection. Else, take $I=(x_1^2,\dots, x_n^2) + \mathrm{RLex}(x_ix_j)$ where $i,j$ is chosen such that it has $\mu$ generators. This can always be done as the reverse lexicographic order is a total order on all monomials. Then $R/I$ has the SLP by Theorem \ref{thm:pf1}.
\end{proof}

\section{The shape of its multiplication matrices}

Having proved that all our ideals do enjoy the strong Lefschetz property when the characteristic of our field $k$ is zero, we want to extend this results to fields of positive characteristic. Using methods from \cite{New_Stanley}, similar to those used earlier in \cite{Watanabe_Boolean}, and our already established results about their Hilbert series, we can prove that it still has SLP, at least if the characteristic is large enough. But first, let us set some notation, inspired by \cite{New_Stanley}. \\

As before, we are looking at $A=R/I$ where $I=(x_1^2,\dots, x_n^2)+\mathrm{RLex}(x_ix_j)$. Here we moreover assume that we are in the interesting case when $j<n$ with $n\geq 3$. Begin by denoting $B_k$ the set of monomials of degree $k$ not divisible by any monomial in $I$, so $B_k$ is a basis for $A_k$. Next, as in the proof of Theorem \ref{thm:pf1}, set $\bar{R}=k[x_1,\dots, x_{n-1}]$, $\bar{I}=(x_1^2,\dots, x_{n-1}^2) + \mathrm{RLex}(x_ix_j)$ and $\bar{A}=\bar{R}/\bar{I}$. If $\bar{B}_k$ is the corresponding basis for $\bar{A}_k$, then we get a decomposition
\[
B_k = \bar{B}_k \sqcup x_n\bar{B}_{k-1}
\]
where we recall that we still are in the case where $j<n$. Further, let $\ell=x_1+\cdots + x_n$ and $\bar{\ell}=x_1+\cdots + x_{n-1}$. If $\bar{M}_i^t$ denotes the matrix representing $\cdot \bar{\ell}^t:\bar{A}_i \to \bar{A}_{i+t}$ with respect to the described bases, then these can be used to determine the matrices corresponding to multiplication by powers of $\ell$ as well.

\begin{lemma}
\label{lem:partition}
For $i,t>0$, the matrix $M_i^t$ representing $\cdot \ell^t:A_i \to A_{i+t}$ with respect to $B_i$ and $B_{i+t}$ can be written as a block matrix
\[ M_i^t=
\begin{pmatrix}
\bar{M}_i^t & 0\\
t\bar{M}_i^{t-1} & \bar{M}_{i-1}^t
\end{pmatrix}.
\]
\end{lemma}

\begin{proof}
We have that
\[
\ell^t = (\bar{\ell} + x_n)^t = \sum_{k=0}^t\binom{t}{k}\bar{\ell}^k x_{n}^{t-k} = \bar{\ell}^t + t\bar{\ell}^{t-1}x_n
\]
in $A$ since $x_n^2=0$ there. The matrix representation then follows from the basis decomposition $B_k = \bar{B}_k \sqcup x_n\bar{B}_{k-1}$.
\end{proof}

To determine the rank of $M_i^t$, we will need a final lemma on ranks of certain block matrices, generalizing a computation done in \cite[Lemma 2.6]{New_Stanley}.

\begin{lemma}
\label{lem:Block}
Let $A,B,P$ be matrices over a field $k$ such that $APB$ is defined and let $\alpha\in k$ be a non-zero element. Then the block matrix
\[ M=
\begin{pmatrix}
AP & 0\\
\alpha P & PB
\end{pmatrix}
\]
satisfies $\mathrm{rank}(M)=\mathrm{rank}(P) + \mathrm{rank}(APB)$.
\end{lemma}

\begin{proof}
This follows directly from the equality
\[
\begin{pmatrix}
I & -\alpha^{-1}A \\
0 & I
\end{pmatrix}
\begin{pmatrix}
AP & 0\\
\alpha P & PB
\end{pmatrix}
\begin{pmatrix}
I & \alpha^{-1}B\\
0 & -I
\end{pmatrix}
=
\begin{pmatrix}
0 & \alpha^{-1}APB\\
\alpha P & 0
\end{pmatrix}
\]
and the fact that $\mathrm{rank}(\alpha E)=\mathrm{rank}(E)$ for any matrix $E$ as $\alpha$ is invertible.
\end{proof}

We can now get to our main result.

\begin{theorem}
Let $A=R/I$ for $I=(x_1^2,\dots, x_n^2)+\mathrm{RLex}(x_ix_j)$ where $1\leq i<j\leq n$ and $R=k[x_1,\dots, x_n]$ where the characteristic of the field $k$ is greater than the socle degree $D$ of $A$; in particular it can be taken as anything greater than $n$. Then $A$ has the strong Lefschetz property.
\end{theorem}

\begin{proof}
This proof will follow the same structure as the proof of SLP for monomial quadratic complete intersections in \cite{New_Stanley}. We already know the statement is true when $j=n$ from Lemma \ref{lem:j=n case}. As those cases cover all of our ideals considered for $n\leq 2$, we may assume that $n\geq 3$ and $1\leq i < j < n$. Moreover, $\cdot \ell^t:A_0 \to A_t$ is injective for any $t\leq D$ as $\ell^t$ is supported at every monomial in degree $t$, and the map is always surjective if $t\geq D$. We are thus left to prove that $\cdot \ell^t:A_i \to A_{i+t}$ has full rank for any $i>0$, $0<t<D$ and $n\geq 3$. We will do this by induction on $n$.

Fix $t$ with $0<t<D$, $i>0$ and $n\geq 3$ and assume $A$ has the SLP for all smaller values of $n$. We then know from Lemma \ref{lem:partition} that the multiplication matrix $M_i^t$ of $\cdot \ell^t:A_i \to A_{i+t}$ can be written as
\[
M_i^t = 
\begin{pmatrix}
\bar{M}_i^t & 0\\
t\bar{M}_i^{t-1} & \bar{M}_{i-1}^t
\end{pmatrix}.
\]
Here we note that $\bar{M}_i^t=\bar{M}_{i+t-1}^1\bar{M}_i^{t-1}$ and $\bar{M}_{i-1}^t =\bar{M}_i^{t-1}\bar{M}_{i-1}^1$ from decomposing $\cdot \bar{\ell}^t:\bar{A}_i \to \bar{A}_{i+t}$ as $\bar{\ell}^t=\bar{\ell}\cdot \bar{\ell}^{t-1}$ and $\cdot \bar{\ell}^t:\bar{A}_{i-1} \to \bar{A}_{i+t-1}$ as $\bar{\ell}^t=\bar{\ell}^{t-1}\cdot \bar{\ell}$. Moreover, since $0<t<D$ and the characteristic of the field is at least $D$, we know that $t$ is non-zero when considered as an element of $k$. Thus $M_i^t$ is a block matrix of the form in Lemma \ref{lem:Block}, so
\[
\mathrm{rank}(M_i^t) = \mathrm{rank}(\bar{M}_i^{t-1}) + \mathrm{rank}(\bar{M}_{i+t-1}^1\bar{M}_i^{t-1}\bar{M}_{i-1}^1).
\]
As we assume that $\bar{A}$ has the SLP by our induction assumption, we know that $\bar{M}_i^{t-1}$ has full rank since it is representing $\cdot \bar{\ell}^{t-1}:\bar{A}_i \to \bar{A}_{i+t-1}$. Similarly, as 
\[
\bar{M}_{i+t-1}^1\bar{M}_i^{t-1}\bar{M}_{i-1}^1 = \bar{M}_{i-1}^{t+1}
\]
is representing $\cdot \bar{\ell}^{t+1}:\bar{A}_{i-1} \to \bar{A}_{i+t}$, it also has full rank. Therefore
\[
\mathrm{rank}(M_i^t) = \min\{|\bar{B}_i|, |\bar{B}_{i+t-1}|\} + \min\{|\bar{B}_{i-1}|, |\bar{B}_{i+t}|\}.
\]
Using that $(|\bar{B}_k|)_{k\in \mathbb{Z}}$ is a mid-heavy sequence from Corollary \ref{cor:mid_heavy}, we finally get
\begin{align*}
\min\{|\bar{B}_i|, |\bar{B}_{i+t-1}|\} + \min\{|\bar{B}_{i-1}|, |\bar{B}_{i+t}|\} &= \min\{|\bar{B}_i| + |\bar{B}_{i-1}|, |\bar{B}_{i+t-1}| + |\bar{B}_{i+t}|\}\\ 
&= \min\{|B_i|,|B_{i+t}|\},
\end{align*}
showing that $M_i^t$ has full rank and that $A$ has the SLP.
\end{proof}

\section{Higher degrees}
Having found an explicit construction for an algebra with the SLP for any number of variables and any possible number of quadratic generators of an artinian ideal in characteristic zero, one may wonder if the same can be done in higher degrees. That is, given $n,d\geq 3$ and some number $\mu \in \left[n,\binom{n+d-1}{d}\right]$, is there an algebra minimally generated by $\mu$ monomials of degree $d$ in $n$ variables with the SLP? We can focus on $n\geq 3$ as all artinian ideals give algebras with the SLP when $n\leq 2$ \cite[Proposition 4.4.]{Codim_2}.  We might start by noting that one possible analogue of our construction from the quadratic case does not work.

\begin{proposition}
Let $R=k[x_1,x_2,x_3]$ for $k$ a field of characteristic zero and let 
\[
I=(x_1^d, x_2^d, x_3^d) + \mathrm{RLex}(x_1x_2^{d-1}) = (x_1,x_2)^d + (x_3^d)
\]
for some integer $d\geq 3$. Then $A=R/I$ does not satisfy the SLP as for $\ell = x_1+\cdots + x_n$, $\cdot \ell^3:A_{d-2} \to A_{d+1}$ fails required injectivity.
\end{proposition}

\begin{proof}
Recall that it suffices to consider the liner form given by the sum of the variables when examining if a monomial ideal has the WLP or the SLP as first proven in \cite[Proposition 2.2.]{Sum_of_var_WLP} in the case of WLP and later in \cite[Theorem 2.2]{Sum_of_var_SLP} for SLP and any field. We begin by showing that $\cdot \ell^3$ must be injective from degree $d-2$ to degree $d+1$ if it has full rank by establishing that $\mathrm{HF}(A,d+1)\geq \mathrm{HF}(A,d-2)$. Set $J=(x_1,x_2)^d$. Then as
\[
A=k[x_1,x_2]/J \otimes_k k[x_3]/(x_3^d),
\]
we can determine the Hilbert function of $A$ as
\[
\HF(A,i) = \sum_{j=0}^i\HF(k[x_1,x_2]/J,j)\HF(k[x_3]/(x_3^d),i-j).
\]
Here we have that
\[
\HF(k[x_1,x_2]/J,j)= 
\begin{cases}
j+1 & \text{ if } j<d\\
0 & \text{ otherwise}
\end{cases}
\]
and $\HF(k[x_3]/(x_3^d),j)=1$ for $j<d$ and zero otherwise. Using this we find that
\[
\HF(A,d+1) - \HF(A,d-2) = (3+4+\cdots + d) - (1+2+\cdots+(d-1)) = d-3,
\]
which is greater than or equal to zero for all $d\geq 3$, showing that injectivity is required.

Introducing the variables $y_1=x_1+x_2$ and $y_2=x_3$, it then suffices to show the polynomial identity
\[
-(y_1+y_2)^3f_d=(d-1)(y_1^{d+1} - (-y_2)^{d+1}) + (d+1)(y_1^dy_2 + y_1(-y_2)^d)
\]
where
\[
f_d=\frac{\partial^2}{\partial y_1 \partial y_2}\frac{y_1^{d+1} - (-y_2)^{d+1}}{y_1+y_2}.
\]
The proof of this identity is a straightforward calculation and is therefore omitted. This then exhibits the polynomial $f_d\in A_{d-2}$ as a kernel of multiplication by $(y_1+y_2)^3=\ell^3$ since 
\[
(d-1)(y_1^{d+1} - (-y_2)^{d+1}) + (d+1)(y_1^dy_2 + y_1(-y_2)^d)\in I.
\]
\end{proof}

When $n\geq 4$ and $d\geq3$, we do not even have the WLP in general. To show that, we do need the theory of inverse systems. Let $R=k[x_1,\dots, x_n]$ act on $S=k[X_1,\dots, X_n]$ via differentiation, i.e. $x_i\circ F=\frac{\partial F}{\partial X_j}$ for $F\in S$. Then for a homogeneous ideal $I$ in $R$, we denote $I^{-1}=\{F\in S \; | \; g\circ F=0, \text{ for all } g\in I\}$ the set of all forms that are annihilated by everything in $I$. The important part for us is now that $\cdot \ell:(R/I)_i \to (R/I)_{i+1}$ has the same rank as $\circ \ell:(I^{-1})_{i+1} \to (I^{-1})_i$.

\begin{proposition}
Let $R=k[x_1,\dots, x_n]$ for $k$ a field of characteristic zero and consider the ideal 
\[
I=(x_1^d,\dots, x_n^d) + \mathrm{RLex}(x_1x_2^{d-1}) = (x_1,x_2)^d + (x_3^d,\dots, x_n^d)
\]
for $n\geq 4$ and $d\geq 3$. Then $A=R/I$ does not have the WLP.
\end{proposition}

\begin{proof}
We begin by noting that it suffices to prove the statement for $n=4$ and $n=5$. All other $n$ then follow by a special case of \cite[Lemma 7.8]{O-sequence} given in \cite[Corollary 3.1]{Forcing_WLP} saying that if an artinian graded algebra $A$ fails the WLP, then so does $A\otimes (k[y_1,y_2]/(y_1^d, y_2^d))$. 

Let us consider the case when $n=4$ first. In this case $\cdot \ell$ fails required surjectivity from degree $2d-3$ to degree $2d-2$. Indeed, to see that surjectivity is needed, we can decompose $A$ into a tensor product similar to the previous proposition as
\[
A=k[x_1,x_2]/(x_1,x_2)^d \otimes_k k[x_3,x_4]/(x_3^d, x_4^d)
\]
where each part has a known Hilbert series. 
Using those one finds that
\begin{align*}
\HF(A,2d-3) &= 1\cdot 2 + 2\cdot 3 + 3\cdot 4 + \cdots + (d-2)(d-1) + (d-1)d + d(d-1)\\ 
&= d(d-1) + \sum_{i=1}^{d-1}i(i+1)\\
&=\sum_{i=1}^d i^2 + \frac{d(d-1)}{2} -d
\end{align*}
while
\[
\HF(A,2d-2) = \sum_{i=1}^d i^2.
\]
Hence $\HF(A,2d-3)\geq \HF(A,2d-2)$ as long as 
\[
\frac{d(d-1)}{2} -d \geq 0,
\]
which is true for $d\geq 3$, showing that surjectivity is needed from degree $2d-3$ to $2d-2$. To show that $\cdot \ell:A_{2d-3} \to A_{2d-2}$ is not surjective, we prove that $\circ \ell: (I^{-1})_{2d-2} \to (I^{-1})_{2d-3}$ is not injective. This follows as 
\[
F=(X_1-X_2)^{d-1}(X_3-X_4)^{d-1}
\]
satisfies that $F\in (I^{-1})_{2d-2}$ and $\ell \circ F=0$.

Let us now turn to $n=5$. In this case we will show that $\cdot \ell:A_{2d-2} \to A_{2d-1}$ does not have full rank by showing that it is neither injective nor surjective, bypassing any need for knowledge about the Hilbert series of $A$. In the case of injectivity, we see that if we set \[
f=\frac{(x_1+x_2)^d - (-x_3)^d}{x_1+x_2+x_3}\cdot \frac{x_4^d - (-x_5)^d}{x_4+x_5},
\]
then
\[
\ell f = (x_1+x_2+x_3)f + (x_4+x_5)f\in I.
\]
Finally, to prove that it is not surjective, we see that
\[
F=(X_1-X_2)^{d-1}(X_3-X_4)^{\left\lfloor \frac{d}{3}\right\rfloor}(X_4-X_5)^{\left\lfloor \frac{d}{3}\right\rfloor}(X_5-X_3)^{\left\lfloor \frac{d}{3}\right\rfloor + r}
\]
where $0\leq r \leq 2$ is the remainder when $d$ is divided by $3$, satisfies that $F\in (I^{-1})_{2d-1}$ and $\ell \circ F=0$. So $\cdot \ell:A_{2d-2} \to A_{2d-1}$ does not have full rank and we are done.
\end{proof}

If one were to try different approaches to find families of ideals with the SLP, several seemingly surprising road blocks may also obstruct the search for such algebras.

\begin{example}
Consider the algebra $A$ defined by the ideal 
\[
I=(x_1^3,\dots, x_8^3) + \mathrm{RLex}(x_2x_3^2) = (x_1^3,\dots, x_8^3) + (x_1^2x_2, x_1x_2^2, x_1^2x_3, x_1x_2x_3, x_2^2x_3, x_1x_3^2).
\]
Then $A=k[x_1,\dots, x_8]/I$ has the SLP, but $k[x_1,\dots, x_8]/(I+(m))$ for any monomial $m$ of degree $3$ outside of $I$ does not even have the WLP. In a similar way, the ideal 
\[
J=(x_1^8, x_2^8, x_3^8, x_1^4x_2^2x_3^2, x_1^3x_2^3x_3^2)
\]
has the SLP, but removing either $x_1^4x_2^2x_3^2$ or $x_1^3x_2^3x_3^2$ as a generator, it no longer does.
\end{example}

These examples aside, using the package MaximalRankProperties \cite{MaxRankProperties} in Macaulay2 \cite{M2}, for any admissible number of minimal generators, we have always been able to find a monomial ideal with that number of generators and the SLP. This makes us believe the following conjecture.

\begin{conjecture}
Let $R=k[x_1,\dots, x_n]$ where $k$ is a field of characteristic zero. Fix $n,d\geq 3$ and $\mu$ in the interval $\left[n,\binom{n+d-1}{d}\right]$. Then there is a monomial ideal $I$ minimally generated by $\mu$ monomials of degree $d$ such that $R/I$ has the SLP.
\end{conjecture}

At this point, we don't even have a suggestion for how such ideals would look like. One related result is that if the ideal is generated by all but at most two monomials in degree $d$, then it has the SLP \cite[Theorem 2]{Nearly_all}. But that is very far from covering all possible number of generators, so a family of ideals conjectured to have the SLP would be very interesting to find.

\section*{Acknowledgements}
The author would like to thank Samuel Lundqvist and Lisa Nicklasson for comments on an earlier draft of this paper and the anonymous referee for their careful reading and valuable comments.

\bibliographystyle{plain}
\bibliography{FSLP_ref}

\begin{thebibliography}{10}

\bibitem{Nearly_all}
Nasrin Altafi and Samuel Lundqvist.
\newblock {Monomial ideals and the failure of the Strong Lefschetz property}.
\newblock {\em Collect. Math.}, 73:383--390, 2022.

\bibitem{Forcing_WLP}
Nasrin Altafi and Samuel Lundqvist.
\newblock Forcing the weak {Lefschetz} property for equigenerated monomial
  ideals.
\newblock {\em Trans. Amer. Math. Soc. Ser. B}, 11:540--566, 2024.

\bibitem{O-sequence}
Mats Boij, Juan Migliore, Rosa~Maria Miró-Roig, Uwe Nagel, and Fabrizio
  Zanello.
\newblock On the shape of a pure {O-sequence}.
\newblock {\em Mem. Amer. Math. Soc.}, 218, 2012.

\bibitem{Petter_Handbook}
Petter Brändén.
\newblock Unimodality, log-concavity, real-rootedness and beyond.
\newblock In {\em Handbook of Enumerative Combinatorics}. Chapman and Hall/CRC,
  2015.

\bibitem{Dao_Nair}
Hailong Dao and Ritika Nair.
\newblock On the {Lefschetz} property for quotients by monomial ideals
  containing squares of variables.
\newblock {\em Commun. Algebra}, 52(3):1260--1270, 2024.

\bibitem{M2}
Daniel~R. Grayson and Michael~E. Stillman.
\newblock Macaulay2, a software system for research in algebraic geometry.
\newblock Available at {\tt math.uiuc.edu/Macaulay2}.

\bibitem{Watanabe_Boolean}
Masao Hara and Junzo Watanabe.
\newblock {The determinants of certain matrices arising from the Boolean
  lattice}.
\newblock {\em Discrete Math.}, 308(23):5815--5822, 2008.

\bibitem{Codim_2}
Tadahito Harima, Juan Migliore, Uwe Nagel, and Junzo Watanabe.
\newblock The weak and strong {Lefschetz} properties for artinian {K-algebras}.
\newblock {\em J. Algebra}, 262:99--126, 2003.

\bibitem{Iarrobino_log}
Anthony Iarrobino.
\newblock {Log-concave Gorenstein sequences}.
\newblock {\em J. Commut. Algebra}, 16(1):25--36, 2024.

\bibitem{Lindsey}
Melissa Lindsey.
\newblock {A class of Hilbert series and the strong Lefschetz property}.
\newblock {\em Proc. Amer. Math. Soc.}, 139(1):79--92, 2011.

\bibitem{Sum_of_var_WLP}
Juan Migliore, Rosa~Maria Miró-Roig, and Uwe Nagel.
\newblock {Monomial ideals, almost complete intersections and the Weak
  Lefschetz property}.
\newblock {\em Trans. Amer. Math. Soc.}, 363:229--257, 2011.

\bibitem{Survey}
Juan Migliore and Uwe Nagel.
\newblock Survey article: A tour of the weak and strong {Lefschetz} properties.
\newblock {\em J. Commut. Algebra}, 5(3):329--358, 2013.

\bibitem{WLP_for_quadratic}
Juan Migliore, Uwe Nagel, and Hal Schenck.
\newblock The weak {Lefschetz} property for quotients by quadratic monomials.
\newblock {\em Math. Scand.}, 126(1):41--60, 2020.

\bibitem{Tran_WLP_paths+cycles}
Hop~D. Nguyen and Quang~Hoa Tran.
\newblock {The weak Lefschetz property of artinian algebras associated to paths
  and cycles}, Preprint 2023.
\newblock {\tt arXiv:2310.14368}.

\bibitem{MaxRankProperties}
Lisa Nicklasson.
\newblock Maximal rank properties, a \emph{Macaualy2} package.
\newblock Available at {\tt
  github.com/LisaNicklasson/MaximalRankProperties-Macaulay2-package}.

\bibitem{Sum_of_var_SLP}
Lisa Nicklasson.
\newblock {The strong Lefschetz property of monomial complete intersections in
  two variables}.
\newblock {\em Collect. Math.}, 69:359--375, 2018.

\bibitem{New_Stanley}
Ho~V.~N. Phuong and Quang~Hoa Tran.
\newblock {A new proof of Stanley’s theorem on the strong Lefschetz
  property}.
\newblock {\em Colloquium Mathematicum}, 173:1--8, 2023.

\bibitem{RRR}
Les Reid, Leslie~G. Roberts, and Moshe Roitman.
\newblock {On complete intersections and their Hilbert functions}.
\newblock {\em Canad. Math. Bull}, 34(4):525--535, 1991.

\bibitem{Stanley}
Richard~P. Stanley.
\newblock {Weyl groups, the hard Lefschetz theorem, and the Sperner property}.
\newblock {\em SIAM J. Algebraic Discrete Methods}, 1(2):168--184, 1980.

\bibitem{Watanabe}
Junzo Watanabe.
\newblock {The Dilworth number of Artinian rings and finite posets with rank
  function}.
\newblock In {\em {Commutative algebra and combinatorics}}, volume~11 of {\em
  Advanced Studies in Pure Math.} Kinokuniya Co. North Holland, Amsterdam,
  1987.

\bibitem{Zanello_log}
Fabrizio Zanello.
\newblock {Log-concavity of level Hilbert functions and pure O-sequences},
  Preprint 2023.
\newblock {\tt arXiv:2210.09447}. To appear in \emph{J. Commut. Algebra}.

\end{thebibliography}

\end{document}